\documentclass[12pt,reqno]{amsart}

\usepackage[utf8]{inputenc}

\usepackage[a4paper,top=3cm,bottom=2cm,left=3cm,right=3cm,marginparwidth=1.75cm]{geometry}

\usepackage{amsmath}
\usepackage{amssymb}
\usepackage{amsthm}
\usepackage{enumitem}
\usepackage{dsfont}
\usepackage{verbatim}
\usepackage{xcolor}
\usepackage{bbm}

\theoremstyle{definition}

\theoremstyle{plain}
\newtheorem{theorem}{Theorem}[section]
\newtheorem{lemma}{Lemma}[section]

\newtheorem{cor}{Corollary}[section]

\theoremstyle{remark}

\newcommand{\R}{\mathbb{R}}
\newcommand{\C}{\mathbb{C}}

\newcommand{\F}{\mathbb{F}}

\usepackage[myheadings]{fullpage}
\usepackage{fancyhdr}
\usepackage{lastpage}
\usepackage{graphicx, wrapfig, setspace}
\usepackage[T1]{fontenc}
\usepackage[english]{babel}

\title{On Incidence bounds with M\"obius hyperbolae in positive characteristic}
\author{Misha Rudnev, James Wheeler}

\newcommand{\bpm}{\begin{pmatrix}}
	\newcommand{\epm}{\end{pmatrix}}

\begin{document}

	\begin{abstract}
		We prove new incidence bounds between a plane point set, which is a Cartesian product, and a set of translates $H$ of the hyperbola $xy=\lambda\neq 0$, over a field of asymptotically large positive characteristic $p$. They improve recent bounds by Shkredov, which are based on using explicit incidence estimates in the early terminated procedure of repeated applications of the Cauchy-Schwarz inequality, underlying many qualitative results related to  growth and expansion in groups. 
		The  improvement -- both quantitative, plus we are able to deal with a general $H$, rather than a Cartesian product -- is mostly due to a non-trivial ``intermediate'' bound on the number of $k$-rich M\"obius hyperbolae in positive characteristic. In addition,  we make an observation that a certain energy-type quantity in the context of $H$ can be bounded via the $L^2$-moment of the Minkowski distance in $H$ and can therefore  fetch the corresponding estimates apropos of the Erd\H os distinct distance problem. 
	\end{abstract}

	\maketitle

	\section{Introduction}
	
	Let $\F$ be a field of characteristic $p$. We are  interested in positive and large (in particular odd) $p$, and the particular case $\F=\F_p$. We will also briefly address the case  $\F=\mathbb R.$
	
	Let $A\subset \F$ be a finite point set with cardinality $|A|$. We identify $A$ with its characteristic function $\mathbbm{1}_A(x)$. Consider the set $H$ of translates of the hyperbola $xy=-1$, in the form
	$$
	y  = a + \frac{1}{b-x}\,:\;(a,b)=h\in H\,.
	$$
	We identify $H$ with the set of  M\"obius transformations 
	\begin{equation}\label{e:def} h(x) = a + \frac{1}{b-x}\,.\end{equation}
	and define 
	$$
	\sigma(A,H):= \sum_{h\in H} \sum_{x\in A} \mathbbm{1}_A(h(x))
	$$
	as the number of incidences between  points in  $A\times A\subset \F^2$ and  hyperbolae in $H$. Our analysis extends trivially to the case when $-1$ in $xy=-1$ be replaced by any other nonzero $\lambda\in \F$.
	
	Clearly, no hyperbola in $H$ can support more than $|A|$ points. In addition, as a trivial example  one can take the hyperbola $xy=-1$, choose an arithmetic progression $A_1$ of the values of $x$, set $A_2=1/A_1$ and $A=A_1\cup A_2$. Translating the hyperbola horizontally by elements of $A_1$ yields other hyperbolae supporting $\gg |A|$ points of $A\times A$. Thus for $|H|\ll |A|$ one cannot do better than the trivial estimate $\sigma(A,H)\gg |A||H|$. But this construction is trivial, for the horizontal translates of the hyperbola do not intersect.
	
	When $|H|$ is essentially bigger than $|A|$ (by which we mean $|H|>|A|^{1+\epsilon}$ for some $\epsilon>0$), one can expect a better than trivial bound for $\sigma(A,H)$. If $\F$ is the field of real (complex) numbers, then one has the Szemer\'edi-Trotter type bound 
\begin{equation}
    \label{e:stb}
	\sigma(A,H) \ll |A|^2+ |H| + (|A|^2|H|)^{2/3}\,,
\end{equation}
	which is nontrivial for $|A|\ll |H|\ll|A|^4$\,. The first instance of a similar bound for a set of points and a set of unit circles was shown by Spencer, Szemer\'edi and Trotter \cite{SST}, who pointed out that the proof of the Szemer\'edi-Trotter point-line incidence theorem allows for replacing affine lines with curves that can be viewed as ``pseudolines'' that would include, in particular, translates of a circle or a hyperbola.
	
	\smallskip
	Above and throughout, we use the Vinogradov symbols $X \gg Y$ to indicate that there is an absolute constant $C>0$ such that $X \geqslant CY$ and similarly define $Y \ll X$. 
	The notations $\lesssim, \gtrsim$ hide, on top of this, powers of $\log(|A||H|)$. 
	
	We will also use the standard sum/product set and representation function notations, such as 
	$$
		r_{AB}(x) :=\left|\left\{(a,b)\in A\times B:ab=x\right\}\right| $$
	and use the word ``energy'' for $L^2$-moment type quantities, such as the standard additive energy over $\F$
	$$E_+(A,B)=\sum_x r^2_{A-B}(x).$$
	Our notations will pertain to non-commutative multiplication in the M\"obius group.
	
	\smallskip
	
	 In positive characteristic, the theorem of Stevens and de Zeeuw \cite{SSFZ} gives a reasonably strong  Szemer\'edi-Trotter type bound for the number of incidences between a Cartesian product point set and a set of affine lines in $\F^2$. The theorem has been responsible for much recent progress in sum-product estimates in positive characteristic. However, unlike the Euclidean Szemer\'edi-Trotter theorem, its proof does not readily enable one to replace straight lines with curves, say translates of the unit circle or hyperbola $xy=-1$.

The last generation of incidence bounds in positive characteristic (that would apply to a wide range of set cardinalities, sufficiently small in terms of $p$) 
has been based on blending a renowned algebraic theorem by Guth and Katz \cite[Theorem 4.1]{GK15} on pairwise line intersections in three dimensions with classical concepts from line geometry in the projective three-space, see e.g. \cite{Ru18}, \cite{Zahl21}, \cite{Ru20}. Heuristically, these incidence bounds can directly embrace only so much nonlinearity as one can feed into the Guth-Katz theorem in question. 

The existence of a {\em qualitatively} nontrivial bound for the quantity $\sigma(A,H)$ has been established in much generality by Bourgain \cite{Bour}, see the forthcoming Theorem \ref{bourg}. However, Bourgain's theorem does not enable one, realistically, to estimate its saving  to the trivial bound $\sigma(A,H)\leqslant |A||H|.$

In contrast, Shkredov in his recent paper \cite{Sh20} has shown that incidence theorems from  \cite{Ru18} and \cite{SSFZ} can be fetched  to yield {\em quantitative} bounds for $\sigma(A,H)$,  given that the set of translates $H$ is also a Cartesian product, see the forthcoming Theorem \ref{th:sh}. An opportunity to do so arose after Shkredov composed the translates of $xy=-1$ with each other three times: compositions of hyperbolae as M\"obius transformations come about naturally as one applies the Cauchy-Schwarz inequality to the quantity $\sigma(A,H)$.

Repeated applications of  Cauchy-Schwarz, yielding some non-trivial saving on every step, owing to the Bourgain-Gamburd $L^2$-smoothing lemma \cite{BG} constitute the basis for the proof of Bourgain's Theorem \ref{bourg}.

This paper strengthens Shkredov's incidence bounds in Theorem \ref{th:sh} and removes the assumption that $H$ be a Cartesian  product. Strengthening comes mostly by proving an ``intermediate'' incidence bound for $A\times A$ and any set of M\"obius hyperbolae, which is then fed into the first two applications of the Cauchy-Schwarz inequality to the quantity $\sigma(A,H)$ in the iterative procedure. Removing the assumption that $H$ be a Cartesian  product is partially due to recognising the appearance of the  $L^2$-moment of the Minkowski distance in the set $H$, arising as the result of Shkredov's triple composition.

Just like \cite{Bour} and \cite{Sh20}, our paper relies crucially on the group structure of the set of linear-fractional transformations. It is inherent therefore in this viewpoint that the plane point set, forming incidences with hyperbolae be a Cartesian product. This, unfortunately, restricts the applicability of our results even to, say  turning the hyperbolae by $45^{\circ}$, which would be interesting apropos of the unit distance count.

	
We proceed by formulating the aforementioned results by Bourgain and Shkredov, followed by the main theorem in this paper.

\begin{theorem}\cite[Bourgain]{Bour}\label{bourg}
		For all $\varepsilon>0$, there is $\delta>0$, as follows. Let  $A\subset\F_p$, $H\subset SL_2(p)$ satisfy the conditions: $1\ll |A|<p^{1-\varepsilon}$, $|H|>|A|^{1+\varepsilon}$, and $|H\cap gS|<|H|^{1-\varepsilon}$ for any proper subgroup $S\subset SL_2(p)$ and $g\in SL_2(p)$.
		For $g=\begin{pmatrix}a&b\\c & d \end{pmatrix}\in SL_2(p)$, let $h\in H$ be identified with the curve $cxy-ax+dy-b=0$.
		
		Then the number of incidences 
		$$\sigma(A,H) <|A|^{1-\delta}|H|.$$
	\end{theorem}
	
	Note that in Theorem \ref{bourg}, $H$ is a general {\em three-parameter}, set of hyperbolae, rather than the two-parameter set of translates of $y=-1/x$ in the next two theorems. Theorem \ref{bourg} says that there is a nontrivial incidence estimate for $\sigma(A,H)$ in the $\F_p$-context, as long as $A$ is essentially smaller than $\F_p$ itself (that is $|A|<p^{1-\varepsilon}$) and the number of M\"obius hyperbolae is essentially greater than $|A|$. It involves an additional assumption that much of $H$ cannot lie in a coset of a  proper subgroup of $SL_2(p)$. This assumption can be weakened to the subgroup being abelian, owing to the recent energy bounds in the affine group by Petridis et al, \cite{Pet20}. 
	
	Shkredov has more recently proved what can be regarded a special case of a quantitative version\footnote{In Shkredov's original formulation the two Cartesian products involve four scalar sets, we quote only a symmetric variant.}  of Theorem \ref{bourg}.
	\begin{theorem}\cite[Shkredov, Theorem 16]{Sh20}
		\label{th:sh}
		Let $A\subseteq \F_p$  and $H=B\times B\subseteq \F_p^2$ be a set of translates of the hyperbola $xy=-1$. Then 
		\begin{equation}\label{e:sigma0}\begin{aligned} \left| \sigma(A,H) - \frac{|A|^2|H|}{p} \right| \;\lesssim \;& \min \left( |A|^{1/2} |H|+ |A|^{3/2}|H|^{3/4} \,,\right. \\
				\\ 
				& \qquad \left.  |A|^{3/4}|H|
				+ |A|^{5/4}|H|^{41/48}\right)\,.\end{aligned}
		\end{equation}
	\end{theorem}
	The first line in the right-hand side of \eqref{e:sigma0} (where the second term is viewed as the main one) provides a nontrivial estimate only for $|H|>|A|^2$, while the second line for $|H|\gtrsim|A|^{12/7}$. The logarithmic factor subsumed by the $\lesssim$ symbol  is present only in the second line of  \eqref{e:sigma0}.
	
	\smallskip
	Our main  result is the following improvement of Theorem \ref{th:sh}.
	\begin{theorem}
		\label{thm:T2}
		Let $A \subset  \F$, with   $|A|<\sqrt{p}$ if $\F$ has positive characteristic $p$.
		For a set of translates $H$ of the hyperbola $y=-1/x$, with $|H|>|A|$, suppose at most 
		$M$ translates $(a,b)\in H$ have the same abscissa or ordinate. Then with 
		$$
		M_1  = \left\{ \begin{array}{lll} M, &\mbox{if } |H|\leqslant |A|^{3/2}\,,\\
			|H|^{2/11}|A|^{8/11} & \mbox{otherwise} \end{array} \right.,$$
		one has the estimate
		\begin{equation}\label{e:sigma1} \sigma(A,H) \;\ll \;  |A|^{1/2} |H|+ |A|^{\frac{6}{5}}|H|^{\frac{4}{5}}
			M_1^{\frac{1}{10}}\,.\end{equation}
		Furthermore, with $$
		M_2  = \left\{ \begin{array}{lll} M, &\mbox{if } |H|\leqslant |A|^{4/3}\,,\\
			|H|^{3/22}|A|^{9/11} & \mbox{otherwise} \end{array} \right.,$$
		one has the estimate
		\begin{equation}\label{e:sigma2}\sigma(A,H) \;\ll \;   |A|^{3/4}|H| + |A|^{11/10}|H|^{17/20}( M_2^{1/10} + |H|^{1/15})
		\end{equation}
		where  $( M_2^{1/10} + |H|^{1/15})$ can be  replaced by $|H|^{1/16}$ if $H=B\times B$.
	\end{theorem}
	Estimate \eqref{e:sigma1} is nontrivial for $H>|A|^{3/2}$, as well as for $|H|>|A|^{4/3}$, assuming $M\leqslant |H|^{1/2}$ as is the case in Theorem \ref{th:sh}. Estimate  \eqref{e:sigma2} generalises and improves  the second line of Shkredov's estimate \eqref{e:sigma0}. It adds to \eqref{e:sigma1} by yielding a better bound for rich hyperbolae. Hence, it is always nontrivial for $|H|>|A|^{4/3}$, as well as for $|H|>|A|^{6/5}$, assuming  $M\leqslant |H|^{2/3}$. 
	
	Other than that, for a few applications we show, we will only use the bound \eqref{e:sigma1}, which is stronger in the most interesting regime $|H|\sim |A|^2$. However, the proof of bound \eqref{e:sigma2} generalising the second line of Shkredov's bound \eqref{e:sigma0} reveals an interesting connection with the Erd\H os distinct distance problem, apropos of the Minkowski distance, naturally associated with hyperbolae. 
	
	\smallskip
	We conclude the introduction by discussing the structure of the proofs.

	The proof of Bourgain's Theorem \ref{bourg} is based on a series of repeated applications of the Cauchy-Schwarz inequality, after each one of which replaces its input $H$, viewed as a set of $SL_2$-transformations, by $H'= H\circ H^{-1}$ (further we just write $HH^{-1}$ and do not make a distinction between $SL_2$ and $PSL_2$). After each Cauchy-Schwarz application, one splits (also by Cauchy-Schwarz) the count into estimating the 
	energy 
	\begin{equation} \label{e:energy} E(H):=\left|\left\{(h_1,h_2,h'_1,h'_2)\in H^4: h_1h_2^{-1}=h'_1{h'_2}^{-1}\right\}\right|\,.\end{equation}
	of the set of M\"obius transformations and, separately, the quantity $\sigma(H',A)$, for which both Bourgain and Shkredov just use a trivial bound. It is estimating the quantity $E(H)$ and its further iterates
	$$T_{k}(H):=\left|\left\{(h_1,\dots h_k,h'_1, \dots h'_k)\in H^{2k}:h_1h_2^{-1}h_3\dots=h'_1{h'}_2^{-1}h'_3\dots\right\}\right|$$
	that underlies nontrivial savings. 
	In Bourgain's proof it comes from repeatedly applying the $L^2$-smoothing lemma of Bourgain and Gamburd \cite{BG}, accumulating very small savings at each step. (For a proof see also \cite[Appendix]{Sh18}.) Each such saving is due to combining  Helfgott's theorem on growth and expansion in $SL_2(p)$ \cite{HH1} with the (non-commutative) Balog-Szemer\'edi-Gowers theorem in $SL_2(p)$. Taking sufficiently many iterations leads to Bourgain's claim.  Getting a quantitative lower bound on the saving $\delta$ in Theorem \ref{bourg} seems forbidding.

	In the proof of Shkredov's Theorem \ref{th:sh}, two applications of Cauchy-Schwarz suffice to yield a quantitative bound for $\sigma(A,H)$, but technically this seems only feasible so far when $H$ is a two-parameter family. (See also \cite{Pet20} and \cite{MW20} for other non-commutative energy estimates.) The main observation is that on the second step of the iteration procedure beginning with the set $H=B\times B$, apropos of  $H'= HH^{-1}HH^{-1}$, one has a much stronger explicit sum-product type $L^2$-estimate quantity  (see the forthcoming Lemma \ref{lem:sh}). Shkredov still used  just a trivial estimate for $\sigma(H',A)$. Fetching these explicit $L^2$ estimates seems to be contingent on the set of translates $H$ being two-dimensional, as we currently -- and regrettably -- do not know a  quantitative estimate for $E(H)$, where $H$ is a (sufficiently small relative to $p$) general set of $SL_2(p)$ transformations. Having such an estimate would directly imply a variant of Helfgott's theorem on growth and expansion in $SL_2(p)$, with much of the machinery that its proof uses made redundant.

	The proof of our Theorem \ref{thm:T2} follows Shkredov, with two main innovations. One is a new ``intermediate'' incidence bound in Theorem \ref{thm:incidence}, which despite being a rather crude corollary of the Stevens - de Zeeuw incidence bound for lines and points, works efficiently even after one application of Cauchy-Schwarz in the iterative procedure, resulting in estimate \eqref{e:sigma1}, which is much better than the first term in the right-hand side of estimate \eqref{e:sigma0}. Furthermore, after one more application of Cauchy-Schwarz we note the connection of the quantity $T_3(H)$ with the $L^2$-moment of the Minkowski distance in the set $H$, see Lemma \ref{lem:T3}. Shkredov's sum-product type bound, see Lemma \ref{lem:sh} is a particular case of this bound when $H=B\times B$ (in which case the estimate is slightly stronger). A sharp Euclidean bound for the quantity in question was central for resolution, by Guth-Katz \cite{GK15} of the Erd\H os distinct distance problem \cite{erdosdistnpoints}; it was adapted to the Minkowski (alias pseudoeuclidean) distance in \cite{RNR}.
	
	We note that in contrast with Bourgain's proof, iterating further would not create new savings for us, since we do not know a way of getting stronger quantitative estimates for $T_n(H)$ with $n>3$, except fetching the $L^2$-smoothing lemma. Moreover, it is inherent in the iteration procedure that the efficiency of using a nontrivial incidence bound for $\sigma(A,H')$ decreases with the number of iterates, as well as that on the $k$th step of the iteration one can only get a nontrivial bound on the number of hyperbolae in $H$, which are $|A|^{1-2^{-k}}$-rich. This is evinced by the very first terms in the right-hand side of estimates \eqref{e:sigma1}, \eqref{e:sigma2}.

	\smallskip

	Finally, we present an adaptation of Theorem \ref{thm:T2} to the case $\F=\F_p$, providing additional estimates, covering the case $|A| \geqslant \sqrt{p}$.
	As the sets $A,H$ get bigger, one should bear in mind the character sum estimate by Iosevich et al. \cite{IHS07}, which is best possible for $|A|^2|H|>p^3$, yet trivial when $|H|<p$:
	$$\sigma(A,H)\leqslant \frac{|A|^2|H|}{p} + 2|A|\sqrt{p|H|}\,.
	$$
	
	We also address in passing the case $\F=\mathbb \R$, as it merely requires a recalculation, using the the sate of the art incidence theorem for modular hyperbolae due to Sharir and Solomon \cite{ShSo}. The reason for doing this is that under the assumption that the number of translates $(a,b)\in H$ with the same abscissa or ordinate is $O(|H|^{1/2})$, then for 
	$|A|\ll |H|\lesssim |A|^{19/13}$,  estimate \eqref{e:sigmaR1} is stronger than the Szemer\'edi-Trotter estimate 
	$\sigma(A,H)\ll |A|^{4/3}|H|^{2/3}$. Moreover, if $|H|\lesssim |A|^{16/13}$ this becomes the case without any assumptions on $H$. (Ideally, of course, one would like to be able to be able to beat the Szemer\'edi-Trotter estimate in the range $|H|\sim |A|^2$.)
	
	\begin{theorem}
		\label{thm:T3}
		Assume the notations of Theorem \ref{thm:T2}. 
		
		Let $\F=\F_p$. If ${\displaystyle  |A||H|^2\leqslant p^3,}$ one can remove the constraint $|A|<\sqrt{p}$ as to \eqref{e:sigma1} and have it with the extra term ${\displaystyle\frac{|A|^{5/4}|H|}{p^{1/4}}}$ in the right-hand side.
		
		Furthermore,\footnote{One can add more intermediate range estimates by fetching more cases from the forthcoming Lemma \ref{lem:T3bd} but they do not  appear to be sufficiently enlightening.} if ${\displaystyle |A||H|^4\leqslant p^5,}$
		one can remove the constraint $|A|<\sqrt{p}$ as to estimate \eqref{e:sigma2} and have it with the extra term ${\displaystyle \frac{|A|^{9/8}|H|}{p^{1/8}}}$ in the right-hand side.
		
		
		If $A\subset \mathbb R$, then
		
		\begin{equation}\label{e:sigmaR1}\sigma(A,H) \;\lesssim  \;  |A|^{1/2}|H| + |A|^{7/6} |H|^{2/3} M_1^{1/6}+ |A|^{\frac{23}{22}}|H|^{\frac{9}{11}}M_1^{\frac{1}{11}}\,.
		\end{equation}

	\end{theorem}

	\section{Some applications} For a point set $P\in \F^2$ and $q=(x,y),\,q'=(x',y')\in P$ we refer to the quantity 
	\begin{equation}\label{eq:Mink}
	D_m(q,q'):=(x-x')^2-(y-y')^2
		\end{equation}
	as the Minkowski distance between $q$ and $q'$.

	\begin{cor}
	    \label{disthypbound}
		Let $A\subset\F$, with $|A+A|,\,|A-A|\leqslant K|A|<\sqrt{p}$. Then the number of realisations of a nonzero Minkowski distance between points of $A\times A$ is  $O(K^{6/5} |A|^{29/10})$.
		\end{cor}
	\begin{proof}
		If $|A+A|,\,|A-A|\leqslant K|A|$ then after the transformation $(x,y)\mapsto(\frac{x+y}{2},\frac{x-y}{2})$, the number of realisations of a nonzero Minkowski distance is bounded via the number of incidences between $(A+A) \,\times \,(A-A)$ with $|A|^2$ translates of the hyperbola $xy=1$.
		
		The claim follows after applying estimate \eqref{e:sigma1} of Theorem~\ref{thm:T2}, with $M_1=|A|$.

	\end{proof}
	Unfortunately, we do not see a way to extend the claim to get an unconditional nontrivial bound for the number of realisations of a single distance, the ``bad'' example being $A=A_1\cup A_2,$ with small $A_1+A_1$ and $A_2+A_2$ but large $A_1+A_2$.
	
	 We are not aware of a nontrivial, that is better than $O(|P|^{3/2}),$ bound on the number of realisations of a nonzero distance between pairs of a point set $P\subset \F_2$ in positive characteristic, where say $|P|<p$. In contrast, a recent paper of Zahl \cite{Zahl21} vindicates the latter exponent $3/2$ for $P$ being a set in {\it three}, rather than {\em two} dimensions (when $-1$ is not a square in $\F$ and $|P|<p^2$).
	 
	 \smallskip
	 Estimate \eqref{e:sigma1} of Theorem \ref{thm:T2} also has the following sum-product type implications.
	 
	 \begin{cor} Let $A\subset \F, $ with $|A|<\sqrt{p}.$ Then 
	 \begin{gather*}
	     |\{(a_1,a_2,a_3,a_4)\in A^4:\, (a_1+a_2)(a_3+a_4)=1\}|\,,\\
	      |\{(a_1,a_2,a_3,a_4)\in A^4:\, (a_1+a_2-a_4)(a_3+a_2 +a_4)=1\}|\,,\\
	       |\{(a_1,a_2,a_3,a_4)\in A^4:\, (a_1+a_2)(a_3+a_2a_4)=1\}|\,,\\
	    |\{(a_1,a_2,a_3,a_4)\in A^4:\, (a_1+a_2+a_4)(a_3+a_2a_4)=1\}|
	 \end{gather*}
	 are all $O(|A|^{29/10}).$
	 
	 Furthermore, if $|A+A|<K|A|<\sqrt{p}$, then the number of points of $A\times A$ on the hyperbola $xy=\lambda\neq 0$ is $O(K^{6/5} |A|^{29/10})$.
	 \end{cor}

	\begin{proof}
	The first group of estimates follows directly by applying estimate \eqref{e:sigma1} with $M_1=|A|$, and $|H|=|A|^2$.
	
	For the last estimate, WLOG $\lambda=-1$; we write
	\begin{multline*}
	 |\{(a_1,a_2)\in A^2:\,a_1a_2=-1\}| \leqslant\\ \frac{1}{|A|^2} |\{(a_1,a_2;s_1,s_2)\in A^2\times (A+A)^2:\,(s_1- a_1)(s_2-a_2)=-1\}|  
	 \end{multline*}and applying estimate \eqref{e:sigma1} to the set $A+A$, with $M_1=|A|$, and $|H|=|A|^2$.
	
		\end{proof}

	\section{Incidence bounds for M{\"o}bius hyperbolae} 
	
	As in the formulation of Theorem \ref{bourg}, a M\"obius Hyperbola is identified with a $SL_2$ (or $PSL_2$) transformation. We say that a hyperbola (transformation) $h$ is $k$-rich if it supports $\geqslant k$ points of $A\times A$, namely
	$|A\cap h(A)|\geqslant k.$

	Over the reals, the best known bound for incidences between points and M\"obius hyperbolae is due to  Solomon and Sharir~\cite{ShSo}.
	
	\begin{theorem}\label{thm:realincidence}
		Let $A\subset \R$ be a set of $n$ real numbers, and consider the set of M\"obius transformations on $\R$, the number of $k$-rich transformations is bounded by
		$$m_k\ll\frac{|A|^4}{k^3}+\frac{|A|^6}{k^{11/2}}\log k\,.$$
	\end{theorem}
	
	We remark that over $\C$ the best known to our knowledge bound is 
	$$ m_k\ll \frac{|A|^6}{k^{5}}\,,$$
	due to Solymosi and Tardos \cite{ST07}.
	
	\smallskip
	We now prove a weaker analogue in positive characteristic, that we have referred to as the intermediate incidence bound.

	\begin{theorem}\label{thm:incidence}
		
		Let $A \subset \F=\F_p$ and $H$ be a set of $m>|A|$ 
		M{\"o}bius hyperbolae in $\F^2$.
		Then the number of incidences between $P=A\times A$ and $H$ satisfies
		
		\begin{equation}
		    \label{eq:hb}
	\sigma(A,H) \ll \frac{|H||A|^2}{p}+{|A|}^{1/2}|H|+\min\left(|A|^{7/5}|H|^{4/5} ,p^{1/3}|A|^{4/3}|H|^{2/3}\right)\,.	\end{equation}
		
		Moreover, if $k>\sqrt{|A|}$ and $|A|< \sqrt{p}$, then for any $\F$ of positive characteristic $p$,
		the maximum number of $k$-rich M\"obius hyperbolae is $O\left(\frac{|A|^7}{k^5}\right)$.
	\end{theorem}
	
	To prove the theorem we will need the following two incidence statements: the Stevens - de Zeeuw theorem \cite{SSFZ} and its corollary from \cite{Murphy2017NEWRO}, specific of $\F=\F_p$.
	
	\begin{theorem} \label{th:sdz} The number of incidences between the point set  $P=A\times A$ and a set $L$ of affine lines in $\F^2$, with $|A||L|<p^2$  is 
		$$\mathcal{I}(P,L)  \ll |A|^{5/4}|L|^{3/4} + |L|+|A|^2\,.$$
	\end{theorem}
	
	\begin{lemma}\label{lem10longpaper}		Let $A \subset \F_p$ and let $2|A|^2/p \leqslant k \leqslant |A|$ be an integer that is greater than
		$1$. The number $l_k$ of $k$-rich lines satisfies
		$$l_k \ll \min \left(\frac{p|A|^2}{k^2},\frac{|A|^5}{k^4}\right).$$
	\end{lemma}
	Note that (non-horizontal and non-vertical) lines in $A\times A$ can be viewed as affine, that is a particular case of M\"obius transformations. This fact underlies the following proof.

	\begin{proof}[Proof of Theorem~\ref{thm:incidence}]
		Let $q=(a,a')\in A\times A=P$ be a point in our point set. Let $H_q\subseteq H$ be the subset of hyperbolae incident to the point $q$. Similarly to above, for $k\geqslant 1$, refer to a hyperbola of $H_q$ as $k$-rich if it supports at least k points of $A\times A$ different from $q=(a,a')$.
		Next, we identify $H_q$ with the set of M\"obius transformations, $M_q$ (that is liner-fractional maps $f(z)=\frac{az+b}{cz+d}$ with $ad-bc\neq0$) mapping $a$ to $a'$. We note $M_q$ lies in a coset of the normal affine subgroup $PSL_2(\F_p)$, the subgroup of upper-triangular matrices. Also for a $g\in M_p$ we have that $\frac{1}{a'-z}\circ g\circ (a-\frac{1}{z})$ maps infinity to infinity.
		
		This shows that the number of incidences between $H_q$ and points in $P$ other than $q$ is equal to the number of incidences between $m=|H_q|$ affine lines and the point set $B\times C:= \frac{1}{a-A}\times (a'-A^{-1})$. Note $|B|=|C|=|A|$.

		 Next, we use Lemma~\ref{lem10longpaper} to bound the number of $k$-rich lines when $\max\left(\frac{2|A|^2}{p},2\right)\leqslant k \leqslant|A|$. Note that unless $\F=\F_p,$ we have taken $|A|<\sqrt{p},$ so $\frac{2|A|^2}{p}\leq 2$. Over $\F_p$, when $2\leqslant k <\frac{2|A|^2}{p}$ one cannot have any nontrivial incidence bound, which accounts for the first term in estimate \eqref{eq:hb}. We also add the trivial Cauchy-Schwarz bound $l_k\leq |A|^4/k^2$.
		
	 Combining this with Lemma~\ref{lem10longpaper} yields the following bound on the number of $k$-rich transformations in $H_q$ by
		$$
		\ll \min \left( \frac{|A|^5}{k^4}, \frac{p|A|^2}{k^2}, \frac{|A|^4}{k^2} \right)\,.
		$$
		The third term in the latter estimate is smaller than the first one if $k\leqslant \sqrt{|A|}$. We proceed, assuming that $k>\sqrt{|A|}$, accounting for the case to the contrary by including the second term in the estimate of the theorem. 
		
		To continue, we sum over $q\in A\times A=P$ observing that we count each $k$-rich hyperbola in $H$ at least $k$ times. This bounds $m_k$, the number of $k$-rich hyperbolae in $H$ as follows:
		
		\begin{equation}m_k\ll\min\left(\frac{|A|^7}{k^5},\frac{p|A|^4}{k^3}\right).\label{eq:bb}\end{equation}

		The proof is concluded by the standard conversion of the latte estimate into an incidence bound. 
		Assuming  $m_k\ll|A|^7/k^5$ and we take some $k=k_*$ and optimise between the estimate $\ll \frac{|A|^7}{k_*^5}$ for the number of incidences, supported on $k_*$-rich hyperbolae and $\ll mk_*$ for the rest of the hyperbolae. Choosing  $k_* = |A|^{7/5}|H|^{-1/5}$ accounts for the first term under the minimum in the theorem's claim. Doing the same thing assuming $m_k\ll\frac{p|A|^2}{k^2}$ accounts for the remaining term and completes the proof in the case $\F=\F_p$.
		
		In the case of general $\F$ we note that once we are only interested in $k\geqslant \sqrt{|A|}$, the constraint $|A|<p$ and the trivial estimate $|A|^4/k^2$ on the number of $k$-rich lines guarantee that the condition $|A||L|<p^2$ of Theorem \ref{th:sdz} is satisfied as to the set $L$ of $k$-rich lines, and hence one has $|L|\ll |A|^5/k^4$ as was used above.
	
	\end{proof}

	\section{Proof of Theorems ~\ref{thm:T2}, \ref{thm:T3}}
We present the proof of Theorem \ref{thm:T2}, making additional remarks in the special case $\F=\F_p$, pertaining to Theorem \ref{thm:T3}: this is when we allow $|A|^2/p\gg 1.$
	
	\begin{proof}
	
		To prove the bound \eqref{e:sigma1} we start out with pruning away the set of translates of the hyperbola that lie on the union of a small number of very rich vertical or horizontal lines. This is done only if $M>|A|$, otherwise at this stage we do nothing. Let $H_1$ denote the  translates, lying on at most $|H|/(x|A|)$, say vertical lines with at least $(x|A|)$ translates per line, for some $x\geqslant 1$. They contribute, trivially, at most $|A||H|/x$ to the quantity $\sigma(A,H)$. Assuming $M =  (x|A|)$ we determine $x$ by setting
		$$
		|A||H|/x = |A|^{6/5}|H|^{4/5}(x|A|)^{1/10}\,,
		$$
		the right-hand side being the bound we will prove in the immediate sequel.
		This means $x = |H|^{2/11}/|A|^{3/11}$. This is exceeds $1$ only if $|H|>|A|^{3/2}$, in which case we have a saving that  determines the choice of $M_1$ apropos of estimate \eqref{e:sigma1}, hence $x|A| = |H|^{2/11}|A|^{8/11}.$ This determines the choice of $M_1$.
		
		We do the same thing concerning  the bound \eqref{e:sigma2}, where we interpolate 
		$$
		|A||H|/x = |A|^{11/10}|H|^{17/20}(x|A|)^{1/10}\,,
		$$
		then $x = |H|^{3/22}/|A|^{2/11}>1$ if $|H|>|A|^{4/3}$, this determines the choice of $M_2$.

		We now move on the proving \eqref{e:sigma1}. Retaining the notation $H$ for the remaining set of translates, 
		apply Cauchy-Schwarz to the summation over $A$, with a shortcut $\sigma=\sigma(A,H)$:
		
		\begin{equation} \label{e:cs1}
			\sigma^2=\left(\sum_{h\in H, a\in A}\mathbbm{1}_A(ha)\right)^2\\
			\leqslant|A|\sum_{u\in HH^{-1}}r_{HH^{-1}}(u)\sum_{ a \in A}\mathbbm{1}_A(ua)
		\end{equation}
		
		Set \begin{equation}\label{e:delt} \Delta:=\frac{\sigma^2}{3|A||H|^2}. \end{equation}
		
		By the pigeonhole principle, since $\sum_{u\in HH^{-1}}r_{HH^{-1}}(u)=|H|^2$,  a positive proportion of the set of incidences is supported on 
		the set $\Omega$ of M\"{o}bius hyperbola, such that for $u\in\Omega$ we have $$\forall u\in \Omega\,,\;\;\;\sum_{a\in A}\mathbbm{1}_A(ua)\geqslant\Delta.$$  
		
		Henceforth we assume $\Delta\gg1$, for otherwise 
		$$\sigma\ll|A|^{1/2}|H|\,,$$
		which accounts for the first term in estimate \eqref{e:sigma1}.
		
		Applying Cauchy-Schwarz to the summation in $u$, restricted to $\Omega$ in \eqref{e:cs1}, yields   
		\begin{equation} \sigma^4\ll|A|^2 E(H)\sum_{u\in\Omega} \left( \sum_{a\in A} \mathbbm{1}_A(ua)\right)^2\,,\label{inequality1}\end{equation}
		where $E(H)$ is the energy, defined by \eqref{e:energy}.
		
		Using formula \eqref{eq:bb}, we conclude that if $\Delta \gg \frac{|A|^2}{p}$, that is unless
		\begin{equation}\label{e:int}
			\sigma\ll  \frac{|A|^{3/2}|H|}{p^{1/2}}\,,
		\end{equation}
		one has (after dyadic summation in $k\geqslant \Delta$ in formula \eqref{eq:bb})
		
		$$\sum_{u\in\Omega}\left(\sum_{a\in A} \mathbbm{1}_A(ua) \right)^2 \ll \min\left( \frac{|A|^7}{\Delta^3},\,\frac{p|A|^4}{\Delta}\right)\,.
		$$
		Observe that from definition of $\Delta$ the minimum being achieved on the second term means that
		\begin{equation}\label{e:int1}
			\sigma\ll  |A||H| \frac{|A|^{1/4}}{p^{1/4}} \,,
		\end{equation}
		which accounts for the corresponding additional term in the statement of Theorem \ref{thm:T3}.
		
		Assuming that the minimum is achieved on the first term
		and applying the bound $E(H)\ll|H|^2M_1$ from the forthcoming Lemma~\ref{lem:T2} (with the quantity $M_1$ having been defined in the pruning procedure at the outset)  gives
		$$\sigma^{10}\ll|A|^{12}|H|^{8}M_1$$
		and completes the proof of estimate \ref{e:sigma1}.
		
		To address the real case in estimate \eqref{e:sigmaR1} in Theorem \ref{thm:T3} we 
		merely recalculate the quantity
		$$\sum_{u\in\Omega}\left(\sum_{a\in A} \mathbbm{1}_A(ua) \right)^2 $$ using Theorem \ref{thm:realincidence}, and estimate \eqref{e:sigmaR1} follows.

		We proceed towards proving estimate \eqref{e:sigma2} by another application of Cauchy-Schwarz  to the summation in $A$ in \eqref{e:cs1}: 
		This yields 
		$$\sigma^4\ll |A|^3\sum_{u\in HH^{-1}HH^{-1}}r_{HH^{-1}HH^{-1}}(u)\sum_{ a \in A}\mathbbm{1}_A(ua).$$
		
		As above, a positive proportion of the set of incidences must be supported on the set $\Omega$ of M\"obius hyperbolae $u$, supporting at least 
		$$\Delta:=\frac{\sigma^4}{3|A|^3|H|^4}\,$$
		points of $A\times A$, thus redefining $\Delta$. We proceed under assumption $\Delta\gg 1$, or else $\sigma\ll|A|^{3/4}|H|.$

		Hence, again by Cauchy-Schwarz, 
		\begin{equation}
			\sigma^8\ll|A|^6 T_4(H)\sum_{u\in\Omega}\left(\sum_{a\in A} \mathbbm{1}_A(ua)\right)^2\,.\label{inequality1'}\end{equation}
		
		Applying \eqref{eq:bb} to estimate the incidence term we assume that the minimum is achieved on its first term, or else by definition of $\Delta$ one has
		$$\sigma \ll (|A||H)(|A|/p)^{1/8}\,,
		$$ which enters the statement of Theorem \ref{thm:T3} in $\F=\F_p$ case.
		
		Hence,
		$$
		\sum_{u\in\Omega}\left(\sum_{a\in A} \mathbbm{1}_A(ua)\right) \ll \frac{ |A|^{16}|H|^{12}}{\sigma^{12}}\,. 
		$$
		Furthermore, by Lemma \ref{lem:T3bd}, we have
		$$T_4(H) \leqslant |H|^2T_3(H) \ll |H|^{6+1/3}$$ in $\F$ of characteristic $p$, with $|H|<p$, and in the specific case 
		$\F=\F_p$ 
		$$
		T_4(H) \leqslant |H|^2T_3(H) \ll |H|^5M_2^2 + \begin{cases}
			\frac{|H|^7}{p},& \text{ if } |H|>p^{5/4}\\
			p^{2/3}|H|^{5+2/3},& \text{ if } p\leqslant|H|\leqslant p^{5/4}\\
			|H|^{6+1/3},&\text{ if } |H|<p
		\end{cases}\,.$$
		Note that the quantity $M_2$, corresponding to the maximum number of translates in $H$, lying on a horizontal/vertical line has been redefined according to the pruning procedure at the outset of the proof.
		
		Combining the last two estimates finishes the proof of Theorems \ref{thm:T2} and  \ref{thm:T3}.  
		

	\end{proof}

	\section{Energy bounds for $H$ }\label{sec:T2}
		This section generalises, from $H$ being a Cartesian product to the general $H$, following statement that Shkredov's Theorem \ref{th:sh} relied on.
	
	\begin{lemma}\cite[Lemma 14]{Sh20} \label{lem:sh}
		For $H=B\times B$  the following estimates hold.
		$$E(H) \leqslant |B|^2E_+(B)\,,$$
		and
		\begin{equation}\label{e:t3sh}T_3(H) \leqslant |B|^2
		\sum_x r^2_{(B-B)(B-B)}(x) + |B|^8\,.\end{equation}
		In positive characteristic, for $|B|<p^{1/2},$ then
		$$
		\sum_x r^2_{(B-B)(B-B)}(x) \lesssim |B|^{5}(E_+(B))^{1/2}\,,
		$$
		if $\F=\F_p$ the  constraint $|B|<p^{1/2},$ can be removed  by adding the extra term $|B|^8/p$ to the right-hand side of the latter estimate.
	\end{lemma}
	
	The generalisation we prove is as follows (we split the statement into two).
	
		\begin{lemma}\label{lem:T2}
		The energy of a set $H$ of translates of the hyperbola $y=-1/x$ is bounded by
		$$E(H)\ll|H|^2M\,,$$
		where $M$ is the maximum number of translates $(a,b)\in H$ having the same abscissa or ordinate.
	\end{lemma}
	
\begin{lemma}\label{lem:T3} Let $H$ be a set of translates of the hyperbola $y=-1/x$. Then
$$ T_3(H) := \sum_{x} r_{HH^{-1}H}^2(x)\leqslant 2|H| Q(H)  + 2|H|^4\,,$$
where 
$$
Q(H)=|\{(h_1,h_2,h_1',h_2')\in H^4:\, D(h_1,h_1')=D(h_2,h_2')\}|\,,
$$
with 
$$D(h,h')= D((a,b),(a',b')') := (a-a')(b-b')\,.$$
\end{lemma}
The quantity $D(h,h')$, with $h,h'\in H$ becomes the Minkowski distance $D_m$, defined be \eqref{eq:Mink} after rotating $H$ by $45^\circ$; in \cite{RNR} quadruples in $Q(H)$ were referred to as {\em rectangular quadruples}, we also use this term in the sequel.

	Hence, the relation \eqref{e:t3sh} in Lemma \ref{lem:sh} is a particular case of the claim of Lemma \ref{lem:T3}. 
	However, the sum-product type bound, following \eqref{e:t3sh} in Lemma \ref{lem:sh} (for proof see e.g. \cite{Sh18})  is slightly stronger than the general one we provide below in Theorem \ref{th:dist}.

	\begin{proof}[Proof of Lemma \ref{lem:T2}]
	The proof is merely mimicking the corresponding part of 
	the proof of Shkredov's Lemma \ref{lem:sh}.
	
	We represent the hyperbola	$ y = a + 1/(b-x)$ by the $SL_2$ matrix $h=\bpm - a & ab+1 \\ -1 & b \epm.$
Without loss of generality we may assume that none of $a,b$ are ever zero.

We have, with $w_1= b_1-b_2$.
$$ h_1h_2^{-1} = \bpm 1+a_1w_1 & a_1-a_2-a_1a_2w_1 \\ w_1 & 1-a_2w_1\epm\,.$$

		Hence, $E(H)$ can be seen to equal the number of solutions to the following set of equations.
		
		\begin{equation}\label{eqsysmtem:T2}
			\begin{aligned}
				a_1 &=a_1',\\
				b_1-b_2&=b_1'-b_2',\\
				a_2&=a_2'
			\end{aligned} \qquad \mbox{or}  \qquad   \begin{aligned}
				b_1 &=b_1',\\
				a_1-a_2&=a_1'-a_2',\\
				b_2&=b_2'\,.
			\end{aligned} 
		\end{equation}
		This completes the proof.
	\end{proof}

\begin{proof}[Proof of Lemma \ref{lem:T3}] 
We will reuse the matrix notation from the previous proof. Note that there is always a trivial bound $T_3(H)\leqslant |H|^2E(H)$, thus from Lemma \ref{lem:T2}  $T_3(H)\leqslant|H|^4M$.

The following argument is somewhat more involved than Shkredov's proof of \eqref{e:t3sh}, where the Cartesian product scenario enables one to easily switch between various $h$'s appearing in the $T_3$ quantity. 

However, the claim one ends up with is in the same spirit: in both estimates for $T_3(H)$ the first term pertains to the rectangular quadruple count, while the second term estimates separately the contribution coming from the Borel subgroup of $SL_2$.

To start we prove the following Lemma to estimate
$$
X_B:= \max_{g \not\in B} \sum_{x\in gB} r^{2}_{HH^{-1}}(x)\,,
$$
where $B$ is the Borel subgroup of upper-triangular matrices.

Let us show that \begin{equation}\label{e:borel}
  X_B\leqslant |H|^2\,.  
\end{equation}

Indeed, a left coset of $B,$ which is not $B$ itself, is defined by a matrix $g=\bpm 1&0\\c&1\epm,$ with $c\neq 0$. 
$$ gB = \bpm s& t \\ cs & ct + s^{-1}\epm\,,\qquad (s,t)\in \F^2,\,s\neq 0\,.$$

Suppose $g_1g_2^{-1} = g_1'(g_2')^{-1}\in gB$.
This means, since $c=0$ that $w_1=w_1'\neq 0$, and therefore, by equating the diagonal entries in $g_1g_2^{-1}$, that
$a_1=a_1'$ and $a_2=a_2'$.

Furthermore, since $c(1+a_1w_1)=w_1$, we can determine $w_1=w_1'$, unless $ca_1=1$ which we will show cannot happen else we reach a contradiction as we have that $w_1=ca_1w_1+c=w_1+c$. So returning to the case $ca_1\neq1$, 
given $g_1=(a_1,b_1)$ and $(a_2',b_2')$ we know $w_1=w_1'$, hence $b_2,b_1$, as well as $a_2,a_1'$. This accounts 
for the claim in estimate \eqref{e:borel}, for  $ca_1=1$ would imply $c=0$.

We are now ready to finish the proof of the claim of Lemma \ref{lem:T3}. Partition
$$
T_3 =  \sum_{x\in B} r_{HH^{-1}H}^2(x)  +   \sum_{x\not \in B} r_{HH^{-1}H}^2(x):= Y_B
+ \overline{Y}_B,
$$
where $Y_B$ is the part of $T_3$, corresponding to $x\in B$, and $ \overline{Y}_B$ the complement.

It follows from \eqref{e:borel} that
$$Y_B\leqslant |H|^4\,.$$

Indeed, let us write
$$
Y_B =  \sum_{x\in B} \left( \sum_{h_3 \in H} r_{HH^{-1}h_3}(x) \right)^2 \leqslant |H|  \sum_{x\in B,\,h_3\in H} 
r_{HH^{-1}h_3}^2(x) \,.$$
Observe that for each value of $h_3$, one has
$$ r_{HH^{-1}h_3}^2(x) =|\{(h_1,h_2,h_1',h_2')\in H^4:\, h_1h_2^{-1}=h_1'(h_2')^{-1}\in h_3^{-1}B\}|\,.$$
Now apply estimate \ref{e:borel}, for each $h_3$, also observe that $h_3\not\in B$.

The Borel case considered above has contributed $|H|^4$ to the bound of the lemma. It remains to estimate the quantity $\overline{Y}_B$. 

By Cauchy-Schwarz,
$$
 \sum_{x \not \in B} r_{HH^{-1}H}^2(x) =  \sum_{x \not \in B} \left ( \sum_{h_2\in H} r_{Hh_2^{-1}H}(x) \right)^2\leqslant 
 |H|\sum_{x \not \in B, h_2\in H} r^2_{Hh_2^{-1}H}(x)\,.
$$
The product appearing in $T_3$ equals
$$
h_1h_2^{-1}h_3 = \bpm  -a_1(w_1w_2+1)-w_2 & 1+a_1w_1+b_3(w_2+a_1(1+w_1w_2)) 
\\ -(1+w_1w_2) & w_1+b_3(1+w_1w_2)\epm\,,
$$
with an extra notation $w_2=a_3-a_2$. In  addition, we have $1+w_1w_2 \neq 0$.
And we have $h_1h_2^{-1}h_3 = h_1'h_2^{-1}h_3',$ with the corresponding notations 
$w_1'=b_1'-b_2$, $w_2'=a_3'-a_2$.

It follows that 
\begin{equation}\label{e:misc}
c=w_1w_2=w_1'w_2',\qquad (a_1'-a_1)c = w_2-w_2' = (a_3-a_3'),\qquad (b_3-b_3')c = b_1'-b_1\,.
\end{equation}
Since $c\neq 0$, this implies
$$
(a_1-a_1')(b_1-b_1') = (a_3-a_3')(b_3-b_3')\,.
$$
This is a rectangular quadruple, with $(a_2,b_2)$ having been eliminated.

It remains to show that given a nontrivial rectangular quadruple, there is only at most two $(a_2,b_2)$, corresponding to it. 
Of course, if the quadruple is trivial, that is $h_1=h_1',\,h_3=h_3'$, then there are $|H|$ choices for $h_2$.

Suppose, we have a fixed nontrivial quadruple $(h_1,h_3,h_1',h_3')$, 
which means from equations \eqref{e:misc} we know $c$. Thus $(a_2,b_2)$ is on the intersection of $H$ with the 
hyperbola $(a_3-x)(b_1-y)=c$, as well as the hyperbola  $(a'_3-x)(b'_1-y)=c$. The intersection is at most two points, 
unless this is the same hyperbola, namely $a_3=a_3',\, b_1=b_1'$.

Furthermore,  from equalising the top right entries of $h_1h_2^{-1}h_3=h_1'h_2^{-1}h_3',$ we have
$b_3w_2 - b_3'w_2' = s$,  where the right-hand side $s$ is known from the quadruple. 
Therefore, we can determine $a_2$, and hence have at most two $h_2$, unless in addition to already 
having $a_3=a_3',\, b_1=b_1'$ we have $b_3=b_3'$. 

But then we have $h_3=h_3'$, and therefore $h_1=h_1'$, so are in the trivial quadruple case. This adds another $|H|^4$ to the bound of the lemma and completes the proof.
\end{proof}

	\subsection{Minkowski distance $L^2$ bound}
	By Lemma \ref{lem:T3}
	\begin{equation}
	T_3(H)\lesssim |H| Q(H)+|H|^4\,,\label{eq:t3}
		\end{equation}
	where $Q(H)$ is the number of rectangular quadruples in $H$. After changing variables $(a,b)\to \left(\frac{a+b}{2},\frac{a-b}{2}\right)$, with $H_m$ replacing $H$ in the new variables one has, with
	now $h_1=(a_1,b_1),\ldots,h_2'=(a_2',b_2')$ in $H_m$, that $Q(H)$ equals the number of solutions of
	$$ (a_1-a_1')^2 - (b_1-b_1')^2 =  (a_2-a_2')^2 - (b_2-b_2')^2\,.
	$$
	Clearly, if $-1$ is a square in $\F$,  one is free to change $-$ to $+$ in the quadratic form, becoming the analogue of the Euclidean distance $\|\cdot\|$  in $\F^2$. 
	
	In this respect, we take advantage of the following result by Murphy et al. in \cite{murphy2020pinned}.

	\begin{theorem}[Theorem 4 \cite{murphy2020pinned}] \label{th:dist}
		Let $H_m\subseteq \F^2_p$. Set
		$$
		Q^*(H_m) = |\{(h_1,h_2,h_1',h_2') \in H_m^4:\,\|h_1-h_1'\| = \|h_2-h_2'\| \neq 0\}|\,.
		$$
		Then 
		$$Q^*(H_m)\ll\begin{cases}
			\frac{|H_m|^4}{p},& \text{ if } |H_m|>p^{5/4}\\
			p^{2/3}|H_m|^{8/3},& \text{ if } p\leqslant|H_m|\leqslant p^{5/4}\\
			|H_m|^{10/3},&\text{ if } |H_m|<p\,.
		\end{cases}$$
		Moreover,  $Q^*(H_m,) \ll |H_m|^{10/3}$, for  $|H_m|<p$ in any field of characteristic $p$.
	\end{theorem}
	Since Theorem \ref{th:dist} allows for $-1$ being a square in $\F$, it applies to the distance $D$ as well. (Besides, in the case $|H_m|<p$ we can always pass to an extension of $\F$.)
	
	Hence, in order to bound the quantity $T_3$ we will just need to add to the above bound on $Q^*(H_m)$ the count of rectangular quadruples in $H$, contributed by the case
	$D(h_1,h_1')=D(h_2,h_2')=0$. If $M$ is the maximum number of points in $H$ on a horizontal or vertical line, their number is trivially at most $M^2|H|^2$. After multiplying by $|H|$ according to \eqref{eq:t3}, this term will dominate the $|H|^4$ term.

	We have therefore established the following lemma.
	\begin{lemma}\label{lem:T3bd}
		For a set $H$ of translates of the hyperbola $y=-1/x$ in $\F_p^2$, such that at most $M$ translates $(a,b)\in H$ have the same abscissa or ordinate, one has
		$$T_3(H)\ll  |H|^3M^2+\begin{cases}
			\frac{|H|^5}{p},& \text{ if } |H|>p^{5/4}\\
			p^{2/3}|H|^{3+2/3},& \text{ if } p\leqslant|H|\leqslant p^{5/4}\\
			|H|^{4+1/3},&\text{ if } |H|<p\,.
		\end{cases}$$
		The bound 
		$$
		T_3(H) \ll |H|^3M^2 + |H|^{4+1/3}\,.
		$$
		holds over a general $\F$ of characteristic $p$, provided that $|H|<p$.
	\end{lemma}
This was all that remained to complete the proof of estimate \eqref{e:sigma2} as well as the $\F_p$ claims of Theorem \ref{thm:T3}.

	\section{Acknowledgements} We thank Ilya Shkredov for encouragement and interest in the approach in this paper. The First Author has been partially supported by the Leverhulme Trust Grant RPG-2017-371.

	\bibliographystyle{plain}
\bibliography{bibliography}

\end{document}